\documentclass{amsart}

\usepackage{amsmath}
\usepackage{amssymb}
\usepackage{amsthm}
\usepackage{graphics}
\usepackage{geometry}

\newtheorem{thrm}{Theorem}[section]
\newtheorem{lemm}[thrm]{Lemma}

\newtheorem{cor}[thrm]{Corollary}

\theoremstyle{definition}
\newtheorem{defn}[thrm]{Definition}

\theoremstyle{remark}
\newtheorem{rem}[thrm]{Remark}

\numberwithin{equation}{section}

\newcommand{\vect}[2]{\left[ \begin{array}{r}{#1}\\{#2}\end{array} \right]}
\newcommand{\mat}[4]{\left[ \begin{array}{rr}{#1}&{#2}\\{#3}&{#4}\end{array} \right]}

\newcommand{\abs}[1]{\lvert#1\rvert}

\begin{document}
\title{Radix and Pseudodigit Representations in $\mathbb{Z}^n$}
\author{Eva Curry}
\address{Department of Mathematics and Statistics, Acadia University, Wolfville, Nova Scotia, Canada B4P 2R6}
\email{eva.curry@acadiau.ca}
\thanks{The author would like to thank Roe Goodman for his extensive help in refining the proof of the radix representation theorem, and Roger Nussbaum for suggesting the $\| \cdot \|'$ norm.}

\date{February 28, 2005}

\commby{Bryna R. Kra}

\begin{abstract}We define radix representations for vectors in $\mathbb{Z}^n$ analogously with radix representations in $\mathbb{Z}$, and give a sufficient condition for a matrix $A:\mathbb{Z}^n \rightarrow \mathbb{Z}^n$ to yield a radix representation with a given canonical digit set.  We relate our results to a sufficient condition given recently by Jeong~\cite{jeo}.  We also show that any expanding matrix $A:\mathbb{Z}^n \rightarrow \mathbb{Z}^n$ will not be too far from yielding a radix representation, in that we can partition $\mathbb{Z}^n$ into a finite number of sets such that $A$ yields a radix representation on each set up to translation by $A^N s$ for some vector $s$ ($N \geq 0$ will vary).  We call the vectors $s$ \emph{pseudodigits}, and call this decomposition of $\mathbb{Z}^n$ a \emph{pseudodigit representation}.\end{abstract}

\maketitle

\section{Introduction}\label{sec_intro}

A number $x$ in some set (such as the whole numbers, integers, or real numbers) is written in a radix representation if there are a base (also called a radix) $\beta$, a finite set of digits $D$, and a finite number $N$ such that $x$ can be written
\[ x = \sum_{j=0}^{N} \beta^j d_j \]
with the $d_j \in D$.  Examples of radix representations for the whole numbers include our standard decimal representation as well as binary representation.  In both these examples, the number is usually written as a string of digits $d_N d_{N-1} \cdots d_1 d_0$, where the position of a digit in the string indicates the power of $\beta$ that the digit is multiplied by.  Thus radix representations are also known as positional number systems.  It is well known that the nonnegative integers can be written in radix representation with any positive integer $\beta \geq 2$ as a base and with digit set $[\beta] = \{0, 1, \ldots, \beta-1\}$.  Matula~\cite{mat} has shown that digits sets can be found to write the integers (both positive and negative) in radix representation if the base satisfies $\abs{\beta} \geq 3$, $\beta \in \mathbb{Z}$.

Radix representations, or positional number systems, can be generalized to higher dimensions.  For arbitrary dimension $n \geq 1$, any matrix with integer entries $A \in M_n(\mathbb{Z})$ defines a linear transformation on $\mathbb{Z}^n$.  We consider expanding matrices in the following sense.
\begin{defn} A \emph{dilation matrix} is a matrix $A$ with integer entries, each of whose eigenvalues $\lambda$ satisfies $\abs{\lambda}>1$.\end{defn}
The determinant of a dilation matrix is an integer $q = \abs{\det{A}}$ with $q \geq 2$.  A dilation matrix gives a mapping of $\mathbb{Z}^n$ into a sublattice of $\mathbb{Z}^n$ with nontrivial cokernel.\footnote{We take the definition of dilation matrix from multivariable wavelet theory.  See \cite[ch.\ 5]{woj} for more on this subject.}

Let $A$ be a dilation matrix and let $\tilde{D}$ be a set of coset representatives for $\mathbb{Z}^n/A(\mathbb{Z}^n)$.  Note that $\tilde{D}$ has $q$ elements.  Any $x \in \mathbb{Z}^n$ is in a unique coset of $\mathbb{Z}^n/A(\mathbb{Z}^n)$, and thus can be written uniquely as
\[ x = A y + r, \]
for some $y \in \mathbb{Z}^n$ and $r \in \tilde{D}$.  This defines a Euclidean algorithm for decomposing any $x_0 \in \mathbb{Z}^n$:
\begin{align*}
x_0 &= A x_1 + r_0\\
x_1 &= A x_2 + r_1\\
x_2 &= A x_3 + r_2\\
    &\vdots
\end{align*}
where the $x_j$ and $r_j$ are unique, given $x_0$.  We say that this algorithm terminates if there exists some $N \geq 1$ such that $x_j = r_j = \mathbf{0}$ for all $j > N$.
\begin{defn} An $n \times n$ matrix $A$ yields a radix representation (for $\mathbb{Z}^n$) with digit set $\tilde{D}$ if the Euclidean algorithm described above terminates for every $x \in \mathbb{Z}^n$.
\end{defn}
That is, if $A$ yields a radix representation with digit set $\tilde{D}$, then for every $x \in \mathbb{Z}^n$ there exists a nonnegative integer $N = N(x)$ such that
\[ x = A x_1 + d_0 = A^2 x_2 + A d_1+ d_0 = \cdots = \sum_{j=0}^N A^j d_j \]
for some sequence of digits $d_j \in \tilde{D}$.  Thus we can represent $x$ by the sequence 
\[ x = (d_n, \ldots, d_1, d_0),\] 
just like positional number systems in one dimension.  As in the case of positional number systems for $\mathbb{Z}$, we call the coset representatives $d_j$ digits, and to call $\tilde{D}$ a digit set.  

Not all dilation matrices yield radix representations.  For example, when $A = 2$ ($n = 1$), with digit set $\tilde{D} = \{0, 1\}$ we can represent every positive integer with a radix representation (writing them in binary), but we cannot represent any negative integers.  When $A$ does not yield a radix representation, we can still find a unique sequence of digits associated to each $x \in \mathbb{Z}^n$ using the algorithm above.  However this sequence will be infinite for vectors for which the algorithm does not terminate.  Section \ref{sec_pseudo} presents an alternative representation in the case that a dilation matrix does not yield a radix representation.  In section \ref{sec_radix}, we give sufficient conditions for a dilation matrix $A$ to yield a radix representation.

The existence of a radix representation depends on the choice of digit set as well as the matrix used as a radix.  A digit set can be any set of coset representatives, so many potential digit sets are possible.  In general, if one takes any fundamental domain $\tilde{F}$ for $\mathbb{Z}^n$ (that is, any set $\tilde{F} \subset \mathbb{R}^n$ congruent to $\mathbb{R}^n / \mathbb{Z}^n$), then $A(\tilde{F}) \cap \mathbb{Z}^n$ will be a set of coset representatives of $\mathbb{Z}^n / A(\mathbb{Z}^n)$, and thus a candidate for a digit set.  For positional notation in one dimension, the digit set is usually taken to be $\{ 0, 1, \ldots, \abs{\beta}\}$.  This is not a satisfactory digit set in the present context, since it will not allow us to represent negative numbers with a radix representation when $\beta$ is positive. We run into problems since the origin $\mathbf{0}$ is on the ``edge'' of the digit set in some sense (we will make this more precise in section \ref{sec_radix}).  In general, choosing the smallest possible coset representatives as digits, and choosing digits that can be generated by the intersection of a convex set with $\mathbb{Z}^n$, seems to produce digit sets that are easier to work with and more likely to yield a radix representation.  Thus in this paper we consider the digit set
\[ D := A(F) \bigcap \mathbb{Z}^n, \]
where $F = [-\frac{1}{2}, \frac{1}{2})^n$.

\section{Pseudodigit Representation}\label{sec_pseudo}

While not all dilation matrices give a radix representation, we give a representation that always holds.
\begin{defn} When an $n \times n$ dilation matrix $A$ does not yield a radix representation with digit set $D = A(F) \cap \mathbb{Z}^n$, but does satisfy Theorem \ref{pseudodigit_rep} below, we say that $A$ yields a pseudodigit representation.
\end{defn}

\begin{thrm}\label{pseudodigit_rep}
Let $A$ be a dilation matrix of dimension $n \geq 1$.  Let $D = A(F) \cap \mathbb{Z}^n$, where $F = [-\frac{1}{2}, \frac{1}{2})^n$.  There exists a finite set $S$ such that every vector $x \in \mathbb{Z}^n$ can be written in exactly one of the following forms:
\begin{enumerate}
\item $x = \sum_{j=0}^{N(x)} A^j d_j$, where $d_j \in D$ for each $j$ and $N(x) \geq 0$;
\item $x = A^{N(x)}s + \sum_{j=0}^{N(x)-1} A^j d_j$, where $d_j \in D$, $s \in S$, and $N(x) \geq 0$.
\end{enumerate}
\end{thrm}
That is, $\mathbb{Z}^n$ can be partitioned into a set of vectors that can be written in a radix representation, and (a finite number of) translates of that set by sets $\{ A^js:\ j \geq 0, s \in S \}$.  Call the elements of $S$ \emph{pseudodigits}.

Let us clarify the notation.  Individual vectors $x \in \mathbb{Z}^n$ may have a radix representation with matrix $A$ and digit set $D$, however we only say that $A$ yields a radix representation if \emph{every} $x \in \mathbb{Z}^n$ has a radix representation (and thus no $x$ has only a pseudodigit representation).  Those vectors which do not have a radix representation will always have a pseudodigit representation instead.  To distinguish between the case where all vectors $x$ have a radix representation, and the case where some vectors $x$ have only a pseudodigit representation instead of a radix representation, we say that $A$ yields a radix representation with digit set $D$, or $A$ yields a pseudodigit representation, respectively.  We will also say that a vector $x \in \mathbb{Z}^n$ has a pseudodigit representation only when $x$ does not have a radix representation.

We must first define a norm under which the dilation matrix $A$ is expanding.  The $l^2$ norm is not sufficient for our purposes, since $\|Ax\|_2$ may not be greater than $\|x\|_2$ for all vectors $x \in \mathbb{R}^n$.  Take, for example, the dilation matrix and vector
\[ A = \left[ \begin{array}{rr}2&-2\\-1&2\end{array} \right], \quad x = \vect{1}{1}. \]
\begin{lemm}\label{lemm_norm}  Let $\| \cdot \|'$ be the norm defined by
\[ \|x\|' = \left(\sum_{j=0}^{\infty} \|A^{-j}x\|_2^2\right)^{1/2}. \]
Then $\| Ax \|' > \| x \|'$ for all $\mathbf{0} \neq x \in \mathbb{Z}^n$.
\end{lemm}
\begin{proof}
We can see that this series converges using the root test and the fact that the spectral radius
\[ r(B) = \max{\{\abs{\lambda}:\ \mbox{$\lambda$ an eigenvalue of $B$} \}} \]
for any matrix $B$ can also be found by the formula~\cite{ste}
\[ r(B) = \lim_{j \rightarrow \infty} \|B^j\|_{op}^{1/j}, \]
where $\| B \|_{op}$ is the operator norm
\[ \|B\|_{op} = \sup{\{\|Bx\|_2:\ x \in \mathbb{R}^n, \|x\|_2=1\}}. \]
Since
\[ (\|A^{-j}x\|_2^2)^{1/j} \leq (\|A^{-j}\|_{op}^2)^{1/j} (\|x\|_2^2)^{1/j}, \]
with $\lim_{j \rightarrow \infty} (\|x\|_2^2)^{1/j} = 1$ for all $x \in \mathbb{R}^n$,
\[ \lim_{j \rightarrow \infty} (\|A^{-j}x\|_2^2)^{1/j} \leq \max{\{\abs{\lambda}^{-2}:\ \mbox{$\lambda$ an eigenvalue of $A$}\}} < 1, \]
so our series does converge, and defines a norm on $\mathbb{R}^n$.  Note that, for $x \neq \mathbf{0}$,
\begin{align*}
\|Ax\|' &= \left( \sum_{j=0}^{\infty} \|A^{-j+1}x\|_2^2 \right)^{1/2} = \left( \sum_{j=-1}^{\infty} \|A^{-j}x\|_2^2 \right)^{1/2} \\
        &= \left( \|Ax\|_2^2 + \sum_{j=0}^{\infty} \|A^{-j}x\|_2^2 \right)^{1/2} > \|x\|'.
\end{align*}
\end{proof}

\begin{rem}  Our norm is similar to the $\| \cdot \|'$ norm used by Lagarias and Wang~\cite{lw1}.  However, they use \[ \| x \|' = \sum_{j=0}^{\infty} \rho^j \|A^{-j}x\|_2 \] where $\rho$ may be any scalar satisfying $1 \leq \rho < \min{\abs{\lambda}}$, with the minimum taken over all eigenvalues $\lambda$ of $A$.  We use the $l^2$ version of this norm, with $\rho=1$, since we will often wish to compare $\| x \|'$ and $\| x \|_2$ for vectors $x \in \mathbb{R}^n$.  In particular, while the traditional definition of a dilation matrix involves a lower bound on the modulus of the eigenvalues, we will be considering lower bounds on the singular values, in which case it will be more natural to consider the $l^2$ norm.
\end{rem}

\begin{proof}[Proof of Theorem \ref{pseudodigit_rep}]
Let 
\begin{align*}
m :&= \min{\{ \|Ax\|':\ \|x\|'=1\}}, \\
M :&= \max{\{ \|Ax\|':\ \|x\|'=1\}}, \quad \mbox{and} \\
\rho :&= \max{\{\|f\|':\ f \in F\}}.
\end{align*}
Then $\|r\|' \leq M\rho$ for every digit $r \in D$.

Given $x \in \mathbb{Z}^n$, let $y \in \mathbb{Z}^n$ and $r \in D$ be the unique vectors such that $x = A y + r$.  Then by the definitions of $m$, $M$, and $\rho$,
\[
\|y\|' = \|A^{-1}(x-r)\|' \leq \frac{1}{m}\|x-r\|' \leq \frac{\|x\|' + \|r\|'}{m} \leq \frac{\|x\|' + M\rho}{m}
\]

If $x$ is big enough, $\|y\|'$ will be less than $\|x\|'$.  A sufficient condition for this is
\[ \|x\|' > \frac{\|x\|' + M\rho}{m}, \]
which simplifies to
\[ \|x\|' > \frac{M\rho}{m-1}. \]
That is, the sequence $(x_0, x_1, x_2, \ldots)$ generated by the multidimensional Euclidean algorithm described above is decreasing until some $x_N$ is in the ball of radius $R:= \frac{M\rho}{m-1}$.  Once inside this ball, the subsequent $x_j$ (for $j > N$) remain in the ball, since in each case $\|x_{j-1}\|' < \frac{M\rho}{m-1}$ and thus
\[ \|x_j\|' \leq \frac{\frac{M\rho}{m-1} + M\rho}{m} = \frac{M\rho}{m-1}. \]

For some vectors $x_0$, the sequence $(x_j)$ will continue to decrease down to the origin.  That is, we can find a (finite) radix representation for these vectors.  For vectors $x_0$ that do not have a radix representation, since the $x_j$ take values in a finite set for $j \geq N$, there will be some $k \geq 1$ and some $l \geq N$ such that $x_{j+k} = x_j$ for all $j \geq l$ (the sequence will start to repeat at some stage after it has entered the ball of radius $R$).  There are a finite number of possible cycles that the sequence $(x_j)$ can fall into, again since there are a finite number of integer vectors inside the ball of radius $R$.  Choose one vector in the cycle to represent each distinct cycle.  These cycle representatives are the pseudodigits. \end{proof}

This suggests a method for determining whether a given dilation matrix $A$ gives a radix representation for every $x \in \mathbb{Z}^n$.  When $A$ is normal, we can write $x = \sum_{\lambda} x_{\lambda}$, where the sum is taken over all eigenvalues $\lambda$ of $A$, and $x_{\lambda}$ is the orthogonal projection of $x$ onto the eigenspace corresponding to the eigenvalue $\lambda$.  We may then write
\begin{align*}
\| x \|'^2 &= \sum_{\lambda} \|x_{\lambda}\|' = \sum_{\lambda} \sum_{j=0}^{\infty} \|A^{-j}x_{\lambda}\|_2^2 = \sum_{\lambda} \sum_{j=0}^{\infty} \|\lambda^{-j} x_{\lambda}\|_2^2 = \sum_{\lambda} \sum_{j=0}^{\infty} (\abs{\lambda}^2)^{-j} \|x\|_2^2\\
           &= \sum_{\lambda} \frac{\abs{\lambda}^2}{\abs{\lambda}^2-1} \|x_{\lambda}\|_2^2,
\end{align*}
and
\[
\| Ax \|'^2 = \sum_{\lambda} \frac{\abs{\lambda}^2}{\abs{\lambda}^2-1} \| \lambda x_{\lambda} \|_2^2 = \sum_{\lambda} \frac{\abs{\lambda}^2}{\abs{\lambda}^2-1} \abs{\lambda}^2 \|x_{\lambda}\|_2^2.
\]
Then $m^2 \|x\|'^2 \leq \|Ax\|'^2 \leq M^2 \|x\|'^2$, with (since $A$ is normal),
\begin{align*}
m &= \min{\{ \abs{\lambda}:\ \mbox{$\abs{\lambda}$ an eigenvalue of $A$}\}},\\
M &= \max{\{ \abs{\lambda}:\ \mbox{$\abs{\lambda}$ an eigenvalue of $A$}\}}.
\end{align*}
We can then compute the radius $R$ and check whether $A$ gives a radix representation of each vector $x \in \mathbb{Z}^n$ with $\|x\|' \leq R$.  

If the dilation matrix $A$ is not normal, it may be difficult to find $R$.  We may use the singular values~\cite{ste} of $A$ to compute an upper bound for the radius $R$, however, provided $\sigma > 1$ for every singular value $\sigma$ of $A$.  Let
\begin{align*}
\mu &= \min{\{ \sigma:\ \mbox{$\sigma$ a singular value of $A$}\}},\\ 
\nu &= \max{\{ \sigma:\ \mbox{$\sigma$ a singular value of $A$}\}}.
\end{align*}
If $\mu > 1$ then
\[ \|Ax\|'^2 \geq \|Ax\|_2^2 \geq \mu^2 \|x\|_2^2 \]
for all $x \in \mathbb{R}^n$, so $m \geq \mu$.  Note also that $\mu^{-1}$ is the maximum of the singular values of $A^{-1}$, with $\mu^{-1} < 1$.  Then $\|A^{-j}x\|_2^2 \leq (\mu^2)^{-j} \|x\|_2^2$ for all $x \in \mathbb{R}^n$, and
\[
\|Ax\|'^2 \leq \sum_{j=0}^{\infty} (\mu^2)^{-j}\|Ax\|_2^2 = \frac{\mu^2}{\mu^2-1} \|Ax\|_2^2 \leq \frac{\mu^2}{\mu^2-1} \nu^2 \|x\|_2^2.
\]
Thus $M \leq \left(\frac{\mu^2 \nu^2}{\mu^2-1}\right)^{1/2}$.  Now we may give a bound on $R$:
\[
R \leq \frac{\left(\frac{\mu^2 \nu^2}{\mu^2-1}\right)^{1/2} \sqrt{n}}{2(\mu -1)} = \frac{\mu \nu \sqrt{n}}{2(\mu-1)^{3/2}(\mu+1)^{1/2}}.
\]

A program for finding digits and pseudodigits for dilation matrices can be obtained by contacting the author.  A few specific examples are given in sections \ref{sec_ex}.

\section{Radix Representation}\label{sec_radix}

As above, set
\begin{align*}
\mu &= \min{\{ \sigma:\ \mbox{$\sigma$ a singular value of $A$}\}},\\
\nu &= \max{\{ \sigma:\ \mbox{$\sigma$ a singular value of $A$}\}}.
\end{align*}
If $\mu$ is large enough, in particular, if $\mu > 2\sqrt{n}$, then $B_{\sqrt{n}} \cap \mathbb{Z}^n \subset D$, where $B_{\sqrt{n}}$ is the ball of radius $\sqrt{n}$ in the $\| \cdot \|_2$ norm centered at the origin.  This means that all vectors of the form
\[ \left[ \begin{array}{r} \epsilon_1 \\ \epsilon_2 \\ \vdots \\ \epsilon_n \end{array} \right], \quad \epsilon_i = -1, 0, 1\ \mbox{for $i = 1,2, \ldots, n$} \]
are contained in the digit set $D$.  As the following theorem shows, this gives a sufficient mathematical formulation of the intuitive idea that we need the origin to not be on the ``edge'' of the digit set in order to ensure that $A$ yields a radix representation.

\begin{thrm}\label{rad_rep}
Let $A$ be an $n \times n$ dilation matrix.  If
\[ \mu > 2\sqrt{n} \]
then $A$ yields a radix representation of $\mathbb{Z}^n$ with the canonical digit set $D = AF \cap \mathbb{Z}^n$, where $F = [-\frac{1}{2}, \frac{1}{2})^n$.
\end{thrm}
\begin{proof}
Let $x_0$ be a vector in $\mathbb{Z}^n$ and let $\{ x_j:\ j \geq 0 \}$ be the orbit of $x_0$ under repeated applications of the Euclidean algorithm.  That is, $x_0 = Ax_1 + r_0$, $x_1 = Ax_2 + r_1$, and so on, with the $r_j$ in our digit set for $j \geq 0$.  We showed in Theorem \ref{pseudodigit_rep} that there exists some $N = N(x_0) \geq 1$ such that either $x_j = r_j = \mathbf{0}$ for all $j > N$, or $x_N$ is a pseudodigit.  Recall that when $x_N$ is a pseudodigit, the orbit of $x_0$ under the Euclidean algorithm falls into a cycle that includes $x_N$, so that there exists an integer $l \geq 1$ and digits $\tilde{r}_0, \ldots, \tilde{r}_{l-1}$ such that 
\[ x_N = A^l x_N + \sum_{i=0}^{l-1} A^i \tilde{r}_i. \]

We use our canonical digit set $D = A(F) \cap \mathbb{Z}^n$, so that for each $j \geq 0$, $r_j = d$ for some $d \in D$.  Given any $x_0 \in \mathbb{Z}^n$ and $N \geq 1$ as above, if $x_j = \mathbf{0}$ for every $j > N$, then we have found a radix representation for $x_0$.  Thus suppose that $x_N$ is a pseudodigit.  Write $x$ in place of $x_N$, since the subscript is no longer necessary.  Note that $x_0 = A^{N-1} x + \tilde{d}$ for some $\tilde{d} \in D_{A,N}$, where
\[ D_{A,j} =\{ \sum_{i=0}^{j} A^id_i:\ d_i \in D \}\ \mbox{for any $j \geq 1$}. \]
Thus $x_0$ has a radix representation if and only if $x$ has a radix representation.

When $x$ is a pseudodigit, there exists an integer $\ell \geq 1$ and a vector $\tilde{d}_x \in D_{A,\ell}$ such that
\[ x = A^{\ell} x + \tilde{d}_x. \]
We rewrite this as
\begin{align*}
A^{-\ell} x &= x + A^{-\ell} \tilde{d}_x \\
       x &= A^{-\ell}x + g_x
\end{align*}
for $g_x = - A^{-\ell} \tilde{d}_x$.  Then
\begin{align*}
x &= g_x + A^{-\ell}x\\
  &= g_x + A^{-\ell}g_x + A^{-2\ell}x\\
  &\vdots \\
  &= \sum_{j=1}^{\infty} A^{-\ell j} g_x = \sum_{j=1}^{\infty} (A^{\ell})^{-j} g_x.
\end{align*}

We can now bound the size of $x$ as follows (since $A^{-lj}x \rightarrow \mathbf{0}$ as $j \rightarrow \infty$).
\[
\| x \|_2 \leq \sum_{j=1}^{\infty} \| (A^{\ell})^{-j} g_x \|_2 \leq \sum_{j=1}^{\infty} \left( \frac{1}{\mu^{\ell}} \right)^{j} \|g_x\|_2 = \frac{1}{\mu^{\ell} - 1} \|g_x\|_2,
\]
with
\begin{align*}
\| g_x \|_2 = \| -g_x \|_2 &\leq \sum_{j=0}^{\ell-1} \|A^{-\ell}A^{j} d_j\|_2 = \sum_{j'=1}^{\ell} \|A^{-j'} d_{\ell-j'}\|_2\\
                           &\leq \sum_{i=0}^{\ell-1} \|A^{-i} f_i\|_2 \quad \mbox{for some $f_i \in F$}\\
                           & \leq \sum_{i=0}^{\ell-1} \frac{1}{\mu^i} \frac{\sqrt{n}}{2} = \frac{\mu^{\ell}-1}{\mu^{\ell}} \frac{\mu}{\mu-1} \frac{\sqrt{n}}{2},
\end{align*}
since $D = A(F) \cap \mathbb{Z}^n$.  Thus
\[
\| x \|_2 \leq \frac{1}{\mu^{\ell} - 1} \frac{\mu^{\ell}-1}{\mu^{\ell}} \frac{\mu}{\mu - 1} \frac{\sqrt{n}}{2} = \frac{\sqrt{n}}{2\mu^{\ell-1}(\mu-1)}.
\]
Note that
\[ \frac{1}{2\mu^{\ell-1}(\mu-1)} < \frac{1}{2} \]
when $\mu > 2\sqrt{n} \geq 2$, thus $x \in B_{\sqrt{n}} \cap \mathbb{Z}^n \subset D$.  That is, $x$ is in fact a digit and not a pseudodigit, and $x_0$ has a radix representation.  Since our choice of $x_0$ was arbitrary, $A$ yields a radix representation for every $x \in \mathbb{Z}^n$, with digit set $D$.
\end{proof}

\begin{rem}\label{rem_eigenvals}
A similar calculation shows that $\|x\|' < \sqrt{n}$ as well; however the $l^2$ bounds are enough to show our result.  We could also replace the hypothesis $\mu > 2\sqrt{n}$ with the slightly weaker hypothesis $m > 2\sqrt{n}$ and follow through the same steps as above using the $\|\cdot\|'$ norm instead of the $\|\cdot\|_2$ norm.  We chose to state the theorem with the hypothesis on $\mu$ since the singular values of a matrix can easily be computed, whereas checking if $m > 2\sqrt{n}$ may be computationally more difficult.
\end{rem}

We note that any dilation matrix $A$ is eventually expanding in the $l^2$ norm, in the following sense.
\begin{cor}\label{big_enough}  For every dilation matrix $A$, there exists a positive integer $\beta \geq 1$ such that $A^{\beta}$ yields a radix representation with digit set $D_{\beta} = A^{\beta}(F) \cap \mathbb{Z}^n$.
\end{cor}
\begin{proof}
In the proof of Theorem \ref{pseudodigit_rep} we saw that 
\[ \sum_{j=0}^{\infty} \|A^{-j}\|_{op}^2 \]
converges (using the root test and the spectral radius formula).  This implies that
\[ \lim_{j \rightarrow \infty} \|A^{-j}\|_{op} = 0. \]
However,
\begin{align*}
\|A^{-j}\|_{op} &= \max{\{\sigma:\ \mbox{$\sigma$ a singular value of $A^{-j}$}\}}\\
                &= \left(\min{\{\sigma:\ \mbox{$\sigma$ a singular value of $A^j$}\}}\right)^{-1}
\end{align*}
for each $j \geq 0$.  Thus, setting
\[ \mu_j = \min{\{\sigma:\ \mbox{$\sigma$ a singular value of $A^j$}\}}, \]
for each $j \geq 1$, we see that
\[ \lim_{j \rightarrow \infty} \mu_j = \infty. \]
In particular, there exists an integer $\beta \geq 1$ such that 
\[ \mu_{\beta} > 2\sqrt{n}. \]
We then apply Theorem \ref{rad_rep} to the dilation matrix $A^{\beta}$.
\end{proof}

By the corollary, we see that we can recapture and improve upon Theorem \ref{rad_rep} using a result of Jeong~\cite{jeo} that the author has recently become aware of.  Jeong defines $C_0$ to be the set of vectors
\[ \left[ \begin{array}{r} \epsilon_1 \\ \epsilon_2 \\ \vdots \\ \epsilon_n \end{array} \right] \]
such that $\epsilon_i = 0$ for all but one index $i = i_0$, with $\epsilon_{i_0} = \pm 1$.  We then set
\[ C = C_0 \bigcup \{ \mathbf{0} \} = B_{1} \bigcap \mathbb{Z}^n. \]
\begin{thrm}[Jeong]\label{thrm_jeong}  Let $M$ be an $n \times n$ matrix with integer entries and $U := \left( -\frac{1}{2}, \frac{1}{2} \right]^n$.  If
\begin{enumerate}
\item $C \subset MU$ (or equivalently $C \subset MU \cap \mathbb{Z}^n$), and
\item $\lim_{j \rightarrow \infty} M^jU = \mathbb{R}^n$,
\end{enumerate}
then $M$ yields a radix representation of $\mathbb{Z}^n$.
\end{thrm}

\begin{cor}\label{best_rad_rep}
Let $A$ be an $n \times n$ dilation matrix.  If
\[ \mu > 2 \]
then $A$ yields a radix representation of $\mathbb{Z}^n$ with the canonical digit set $D = AF \cap \mathbb{Z}^n$, where $F = [-\frac{1}{2}, \frac{1}{2})^n$.
\end{cor}

\begin{proof}  Note that using Jeong's fundamental domain $U$ is equivalent to using our fundamental domain $F$.  If $A$ is a dilation matrix, then $\lim_{j \rightarrow \infty} A^j F = \mathbb{R}^n$ by Corollary \ref{big_enough}.  If $\mu > 2$, then $B_{1} \subset AB_{1/2}$.  Since $B_{1/2} \subset F$ (equivalently, $B_{1/2} \subset U$) and $C \subset B_{1}$, this implies that $C \subset AF$ (equivalently, $C \subset AU$).  The result then follows from Theorem~\ref{thrm_jeong}.
\end{proof}

\begin{rem}
Consider the dilation matrix $A = \textrm{diag}\{2, 2, \ldots, 2\}$, the $n \times n$ matrix with every diagonal element $2$ and every non-diagonal element $0$.  As in the first example in the following section, this matrix does not yield a radix representation of $\mathbb{Z}^n$.  However, for this matrix $\mu = 2$, and indeed every eigenvalue $\lambda = 2$.  Thus the sufficient condition $\mu > 2$ in Corollary~\ref{best_rad_rep} is sharp.
\end{rem}

\begin{rem}
Unlike in the proof of Theorem~\ref{rad_rep}, we cannot use the $\| \cdot \|'$ norm to replace the condition $\mu > 2$ with the same condition on the eigenvalues of the dilation matrix $A$.
\end{rem}

\section{Examples}\label{sec_ex}

\subsection{Pseudodigit Representations in One Dimension}
When $n=1$ and $A=2$, the digit set $D = A(F) \cap \mathbb{Z}$ is
\[ D = [-1,1) \cap \mathbb{Z} = \{ 0, -1 \}. \]
Note that this is not the standard digit set $\tilde{D} = \{0,1\}$ for binary numbers, so we are able to represent negative integers rather than positive integers with a radix representation.  In this case, $m = M = 2$ and $R = 1$.  Thus we only need check the set $\{ -1, 0, 1 \}$ to find the set $S$ of pseudodigits for $A$.  It is easy to see that $-1$ and $0$ have radix representations.
\begin{align*}
-1 &= 2^{0}(-1)\\
0 &= 2^{0}(0)
\end{align*}
However, if we apply the Euclidean algorithm to $1$, we find that
\begin{align*}
1 &= 2(1) + (-1)\\
1 &= 2(1) + (-1)\\
 \vdots&
\end{align*}
The algorithm does not terminate, but repeats through the cycle $(1)$ (of length one).  Thus we have one pseudodigit.  Since $1$ is the only element in this cycle, we will take $1$ to be our cycle representative.  Then $S = \{1\}$, and $A=2$ yields a pseudodigit representation.

\subsection{Radix Representations in One Dimension}
If we take $A = -2$, then again $D = \{0,-1\}$ and $R = 1$.  In this case, however, every $x$ within the ball of radius $R = 1$ centered at $0$ has a radix representation:
\begin{align*}
-1 &= (-2)^{0}(-1)\\
0 &= (-2)^{0}(0)\\
1 &= (-2)^{1}(-1) + (-2)^{0}(-1)\\
\end{align*}
Thus $A=-2$ yields a radix representation.

When $A = \beta \in \mathbb{Z}$ with $\abs{\beta} \geq 3$, Theorem \ref{pseudodigit_rep} shows that $A$ yields a radix representation for $\mathbb{Z}$ with our standard digit set $D$.

\subsection{The Twin Dragon}
The twin dragon matrix is the matrix\footnote{The twin dragon matrix gets its name from the shape of the set \[ T = \{ \sum_{j=1}^{\infty} A^{-j}d_j:\ d_j \in D\}. \]  For more information, see \cite{woj}.}
\[ A = \mat{1}{1}{-1}{1}. \]
It is one of the smallest examples of an expanding $2 \times 2$ matrix, and dilation by the twin dragon is analogous to dilation by $2$ in many ways.  The set $A([-\frac{1}{2},\frac{1}{2})^2)$ is shown in figure \ref{fig:A(F^2)}.

\begin{figure}[t]
\begin{center}
\includegraphics{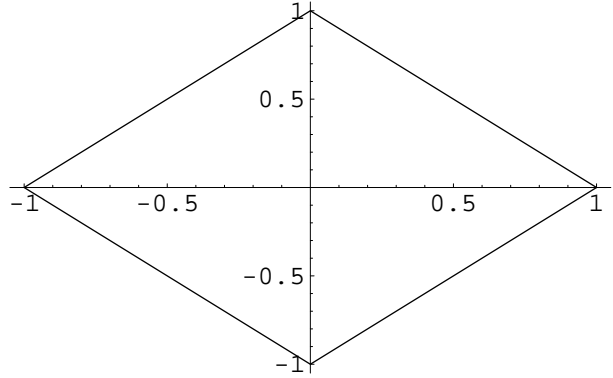}
\end{center}
\caption{$A(F)$ in two dimensions}\label{fig:A(F^2)}
\end{figure}

The digit set $D$ for the twin dragon is the intersection of this set with $\mathbb{Z}^2$:
\[ D = \{d_0 = \vect{0}{0}, d_1 = \vect{-1}{0}\}. \]
(Note that $\det{A} = 2 = |D|$.)  As in the case of dilation by $2$, we have only two digits, and we need one pseudodigit for a pseudodigit representation of all vectors in $\mathbb{Z}^2$:
\[ s = \vect{0}{1}. \]
We can then partition $\mathbb{Z}^2$ into ``negative'' and ``positive'' vectors (those with radix representations and those with pseudodigit representations, respectively), as in the case of dilation by $2$.  Figure \ref{fig:td_both_reps} shows the set of points representable with a radix representation $\sum_{j=0}^{N}A^jd_j$ with $N \leq 6$ and $d_j \in D$ together with the set of points representable with a pseudodigit representation $A^ns + \sum_{j=0}^{N-1} A^jd_j$ with $N \leq 6$ and $D_j \in D$.  

\begin{figure}[b]
\begin{center}
\includegraphics{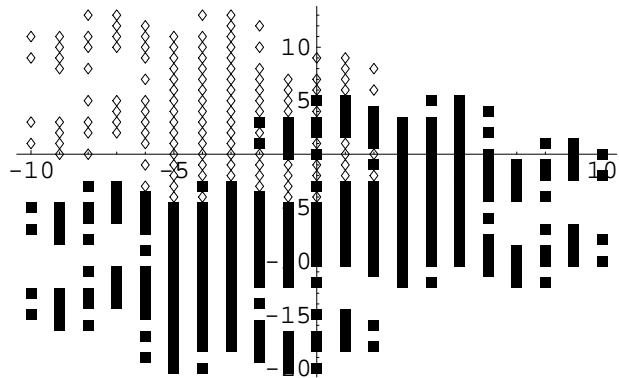}
\end{center}
\caption[Radix and pseudodigit representations with $N=6$.]{Radix and pseudodigit representations with $N=6$.  Boxed point are those with a radix representation.  Diamond points are those with a pseudodigit representation.}\label{fig:td_both_reps}
\end{figure}

In order to satisfy the sufficient condition for a radix representation of $\mathbb{Z}^2$, we must square the twin dragon matrix:
\[ B = A^2 = \mat{0}{2}{-2}{0}. \]
Then $\det{B}=4$ and the eigenvalues of $B$ are $\lambda_1=i\sqrt{2}$ and $\lambda_2=-i\sqrt{2}$, with $|\lambda_1| = |\lambda_2| = \sqrt{2}$.

\subsection{A Higher Dimensional Example}
Another dilation matrix of interest in the study of multivariable wavelets is the following matrix.\footnote{This matrix, found by Lagarias and Wang \cite{lw2}, is the smallest known example of a dilation matrix for which a Haar-like scaling function for a wavelet set cannot be constructed.}
\[ A = \left[\begin{array}{rrrr} 0 & 1 & 0 & 0 \\ 0 & 0 & 1 & 0 \\ 0 & 0 & -1 & 2 \\ -1 & 0 & -1 & 1 \end{array}\right]. \]
This matrix has determinant $\det{A} = 2$ and eigenvalues $\lambda_1=-\sqrt{-\frac{1}{2}-\frac{i\sqrt{7}}{2}}$, $\lambda_2=-\sqrt{-\frac{1}{2}+\frac{i\sqrt{7}}{2}}$, $\lambda_3=\sqrt{-\frac{1}{2}+\frac{i\sqrt{7}}{2}}$, and $\lambda_4=\sqrt{-\frac{1}{2}-\frac{i\sqrt{7}}{2}}$. The digit set for this matrix is
\[ D = \{ d_0 = \left[\begin{array}{r}0\\0\\0\\0\end{array}\right], d_1 = \left[\begin{array}{r}0\\0\\-1\\-1\end{array}\right] \}. \]
In this example, there are two pseudodigits:
\[ S = \{ s_1 = \left[\begin{array}{r}0\\1\\0\\0\end{array}\right], s_2 = \left[\begin{array}{r}-1\\0\\0\\0\end{array}\right] \}. \]

\section{General Lattices}\label{sec_comm}

Any $n$-dimensional point lattice $\Gamma$ in $\mathbb{R}^n$ can be expressed as $M(\mathbb{Z}^n)$ for some matrix $M \in M_{n}(\mathbb{R})$, where $M$ has full rank, and is thus invertible \cite{cos}.  That is, any $n$-dimensional lattice $\Gamma$ in $\mathbb{R}^n$ is isomorphic to $\mathbb{Z}^n$.  Similarly, we can consider point lattices $\Gamma$ in $\mathbb{C}^n$, where $\Gamma = L\mathbb{Z}^n$ for some invertible matrix $L \in M_{n}(\mathbb{C})$, and thus is also isomorphic to $\mathbb{Z}^n$.  

Let $\Gamma$ be an $n$-dimensional point lattice in $\mathbb{R}^n$ or in $\mathbb{C}^n$, as above.  Let $e_1, \ldots, e_n$ be the canonical basis for $\mathbb{Z}^n$: for each $i = 1, \ldots, n$, $e_i$ has $i^{th}$ entry $1$ and all other entries $0$.  Then $f_i := Me_i$ (or $f_i := Le_i$), $i = 1, \ldots, n$, is a basis for $\Gamma$.  For each $x$ in $\Gamma$, we may write
\[ x = x_1 f_1 + x_2 f_2 + \cdots + x_n f_n, \]
for some coefficients $x_i \in \mathbb{Z}$, $i=1, \ldots, n$.  Under the norm
\[ \|x\|_{l^2(\Gamma)} := \left( \sum_{i=1}^{n} |x_i|^2 \right)^{1/2}, \]
the lattice $\Gamma$ is isometrically isomorphic to $\mathbb{Z}^n$.

We call an $n \times n$ matrix $A$ a dilation matrix for the lattice $\Gamma$ if it gives a mapping $A: \Gamma \rightarrow \Gamma$ with nontrivial cokernel.  That is, $A$ is a dilation matrix for $\Gamma$ if $A=MB$ for some dilation matrix $B$ (for $\mathbb{Z}^n$), where $M$ is the matrix giving the isomorphism between $\mathbb{Z}^n$ and $\Gamma$, $\Gamma = M\mathbb{Z}^n$.  Notice that $A$ yields a radix representation of $\Gamma$ if and only if $B$ yields a radix representation of $\mathbb{Z}^n$.  The digit set for this radix representation can be written as $\Gamma \cap A(\tilde{F})$ for some fundamental domain $\tilde{F}$ for the lattice $\Gamma$.  Define $F_{\Gamma}$ to be the specific fundamental domain
\[ F_{\Gamma} = MF = M([-\frac{1}{2},\frac{1}{2})^n). \]

Let $A$ be a dilation matrix for the lattice $\Gamma$.  Corresponding to the norm $\| \cdot \|'$ on $\mathbb{Z}^n$, we may define the norm
\[ \| \cdot \|_{\Gamma}' := \sum_{j=0}^{\infty} \|A^{-j} \cdot \|_{l^2(\Gamma)} \]
on $\Gamma$.  We can now restate our results in greater generality.

\begin{thrm}[Pseudodigit Representation]
Let $\Gamma = M\mathbb{Z}^n$ be a lattice in $\mathbb{R}^n$ or $\mathbb{C}^n$.  Let $A$ be a dilation matrix for $\Gamma$.  Let $D = A(F_{\Gamma}) \cap \Gamma$.  There exists a finite set $S \subset \Gamma$ such that every vector $x \in \Gamma$ can be written in exactly one of the following forms:
\begin{enumerate}
\item $x = \sum_{j=0}^{N(x)} A^j d_j$, where $d_j \in D$ for each $j$ and $N(x) \geq 0$;
\item $x = A^{N(x)}s + \sum_{j=0}^{N(x)-1} A^j d_j$, where $d_j \in D$, $s \in S$, and $N(x) \geq 0$.
\end{enumerate}
\end{thrm}

Let
\[ \mu' := \min{\{ \|Ax\|_{l^2(\Gamma)}:\ x \in \Gamma, \|x\|_{l^2(\Gamma)} = 1 \}}. \]
Note that $\mu$ is equal to the smallest singular value of $B = M^{-1}A$ (where $B$ is a dilation matrix for $\mathbb{Z}^n$).

\begin{thrm}[Radix Representation]
Let $\Gamma = M\mathbb{Z}^n$ be a lattice in $\mathbb{R}^n$ or $\mathbb{C}^n$.  Let $A$ be a dilation matrix for $\Gamma$.  If $\mu' > 2$ then $A$ yields a radix representation of $\Gamma$ with digit set $D = A(F_{\Gamma}) \cap \Gamma$.
\end{thrm}

As a result of the Radix Representation Theorem, we also have the following result.

\begin{lemm}
For every dilation matrix for $\Gamma$, $A$, there exists a positive integer $\beta \geq 1$ such that $A^{\beta}$ yields a radix representation of $\Gamma$ with digit set $D_{\beta} = A^{\beta}(F_{\Gamma}) \cap \Gamma$.
\end{lemm}

\bibliographystyle{amsplain}

\begin{thebibliography}{99}

\bibitem{cos} J.H.\ Conway, N.J.A.\ Sloane, ``Sphere Packings, Lattices, and Groups.''  Springer-Verlag, New York, New York, 1988.

\bibitem{jeo} E.-C.\ Jeong, A Number System in $\mathbb{R}^n$, \emph{J.\ Korean Math.\ Soc.\ } \textbf{41} (2004), 945-955.

\bibitem{lw1} Jeffrey C.\ Lagarias and Yang Wang, Self-Affine Tiles in $\mathbf{R}^n$, \emph{Adv.\ Math.\ } \textbf{121} (1996), 21-49.

\bibitem{lw2} Jeffrey C.\ Lagarias and Yang Wang, Haar Bases for $L^2(\mathbb{R}^n)$ and Algebraic Number Theory, \emph{J.\ Number Theory} \textbf{57} (1996), 181-197.

\bibitem{lw3} Jeffrey C.\ Lagarias and Yang Wang, Corrigendum/Addendum: Haar Bases for $L^2(\mathbb{R}^n)$ and Algebraic Number Theory, \emph{J.\ Number Theory} \textbf{76} (1999), 330-336.

\bibitem{mat} David W.\ Matula,  Basic Digit Sets for Radix Representation, \emph{J.\ ACM} \textbf{29} (1982), 1131-1143.

\bibitem{odl} A.M.\ Odlyzko, Non-Negative Digit Sets in Positional Number Systems, \emph{Proc.\ Lond.\ Math.\ Soc.\ } \textbf{37} (1978), 213-229.

\bibitem{ste} G.W.\ Stewart, ``Matrix Algorithms Volume 1: Basic Decompositions'' SIAM, Philadelphia, Pennsylvania, 1998.

\bibitem{woj} P.\ Wojtaszczyk, ``A Mathematical Introduction to Wavelets.''  Cambridge University Press, Cambridge, United Kingdom, 1997.


\end{thebibliography}

\end{document}